\documentclass{article}
\usepackage{amsmath,amssymb,mathrsfs}
\usepackage{amsthm}
\usepackage{mathtools}
\usepackage{graphicx}

\usepackage{xcolor}
 \usepackage{todonotes}

\newtheorem{theorem}{Theorem}[section]
\newtheorem{proposition}[theorem]{Proposition}

\theoremstyle{definition}
\newtheorem{definition}[theorem]{Definition}

\theoremstyle{remark}
\newtheorem{remark}[theorem]{Remark}

\usepackage[bookmarks=true]{hyperref}
\hypersetup{
    colorlinks,
    linkcolor={red!50!black},
    citecolor={blue!50!black},
    urlcolor={blue!80!black},
}

\DeclareMathOperator{\supp}{supp}

\newcommand{\N}{\mathbb{N}}
\newcommand{\R}{\mathbb{R}}

\newcommand{\diff}{\mathrm d} 
 
\newcommand{\capa}{\mathrm{Cap}}

\newcommand{\BE}{\mathsf{BE}}
\newcommand{\RCD}{\mathsf{RCD}}

\newcommand{\Ric}{\mathrm{Ric}}

\newcommand{\Scal}{\mathcal{S}}
\newcommand{\Rcal}{\mathcal{R}}
\newcommand{\Dcal}{\mathcal{D}}

\newcommand{\Ucal}{\mathcal{U}}

\newcommand{\Acal}{\mathcal{A}}

\newcommand{\di}{\mathsf d} 

\newcommand{\Ascr}{\mathscr A}

\renewcommand{\epsilon}{\varepsilon}
\renewcommand{\phi}{\varphi}

\newcommand\pder[2][]{\ensuremath{\frac{\partial#1}{\partial#2}}}

\newcommand\ppder[2][]{\ensuremath{\frac{\partial^2#1}{\partial#2 ^2}}}

\newcommand{\x}{\mathbf x}
\newcommand{\y}{\mathbf y}

\newcommand{\cdc}{\Gamma}
\newcommand{\comma}{\ ,}
\newcommand{\fstop}{\ .}
\newcommand{\cquad}{\comma \quad}

\newcommand{\1}{\mathbf 1}

\newcommand\blfootnote[1]{
    \begingroup
    \renewcommand\thefootnote{}\footnote{#1}
    \addtocounter{footnote}{-1}
    \endgroup
}

\providecommand{\keywords}[1]
{
  \small	
  {\textit{Keywords---}} #1
}

\title{A phase transition in the Bakry--\'Emery gradient estimate for Dyson Brownian motion} 
\author{Kohei Suzuki\footnote{Department of Mathematical Science, Durham University. \\ \textit{E-mail}:  \href{mailto:kohei.suzuki@durham.ac.uk}{kohei.suzuki@durham.ac.uk}} \ and Kenshiro Tashiro\footnote{Okinawa Institute of Science and Technology (OIST). \\ \textit{E-mail}:  \href{mailto:kenshiro.tashiro@oist.jp}{kenshiro.tashiro@oist.jp}}}
\date{}

\begin{document}

\maketitle
 \blfootnote{\\ \keywords{Dyson Brownian motion, Ricci curvature, 
Bakry--\'Emery gradient estimate}}
\begin{abstract} 
\noindent In this paper, we find a gap between the lower bound of the Bakry--\'Emery $N$-Ricci tensor~$\Ric_N$ and the Bakry--\'Emery gradient estimate $\BE$ in the space associated with the finite-particle Dyson Brownian motion (DBM) with inverse temperature $0<\beta<1$.
Namely, we prove that, for the weighted space $(\R^n, w_\beta)$ with $w_\beta=\Pi_{i<j}^n |x_i-x_j|^\beta$ and any $N\in[n+\frac{\beta}{2}n(n-1),+\infty]$,  
\begin{itemize}
\item $\beta \ge 1 \implies \Ric_N\geq 0 \ \& \ \BE(0,N)$ hold;
\item $0<\beta<1$ $\implies$ $\Ric_N\geq 0$ holds $\text{while $\BE(0,N)$ does not hold}$,
\end{itemize}
which shows a phase transition of the Dyson Brownian motion regarding the Bakry--\'Emery curvature bound in the small inverse temperature regime. 
\end{abstract}
\section{Introduction}
In the seminal paper~\cite{BakEme85}, it was discovered that, for a complete weighted Riemannian manifold~$(M,g,e^{-V}{\rm vol}_g)$ with a potential~$V \in C^2(M)$, the following conditions are equivalent:
\begin{itemize}
    \item The Bakry--\'Emery $N$-Ricci curvature is bound below by a constant  $K \in \R$: 
\begin{align}\tag*{$\Ric_N \ge K$}\label{e:RLB}
\left[{\rm Ric}+{\rm Hess}(V)-\frac{dV\otimes dV}{N-n}\right]_\x(v, v) \ge Kg_\x(v, v)  
\end{align} 
\item The {\it $\BE(K,N)$ gradient estimate} holds: 
\begin{align}\label{BE}|\nabla T_tu|^2 + \frac{1-e^{-2Kt}}{NK}({\rm L}T_t u)^2\le e^{-2Kt}T_t|\nabla u|^2 \comma \tag*{$\BE(K,N)$} 
\end{align}
where $\{T_t\}_{t \ge 0}$ is the semigroup whose infinitesimal generator is ${\rm L}=\frac{1}{2}\Delta - \nabla V\cdot\nabla$ and $u\in H^{1,2}(M)$ is a function in the $(1,2)$-Sobolev space~on~$M$. 
By convention, the second term in the LHS is 
\[\frac{1-e^{-2Kt}}{NK}=\begin{cases}
    0 & N=+\infty,\\
    \frac{2t}{N} & K=0.
\end{cases}\]
\end{itemize}
Remarkably, the \ref{BE} does not require much regularity of spaces nor curvature tensors, which, therefore, opened a way to speak about the condition~``Ricci curvature $\ge K$ and the dimension $\le N$" in singular spaces without curvature tensors such as weighted Riemannian manifolds, infinite-dimensional spaces, and metric measure spaces, see e.g., \cite{BakGenLed14} for a comprehensive reference.

\smallskip
In this paper, we reveal a gap between \ref{e:RLB} and~\ref{BE} in a particular weighted manifold whose potential $V$ is singular, thus violates the condition $V \in C^2(M)$.
Let~$w_\beta(\mathbf x)=\prod_{i < j}|x_i-x_j|^\beta$ for $\mathbf x= (x_1,\ldots, x_n)$ and $\beta > 0$. Consider the weighted manifold $(\R^n, g,w_\beta)$ with the standard Euclidean metric~$g$, which is called the \emph{Dyson space} in this paper.
In the context of statistical physics and random matrices, the constant $\beta$ is called {\it inverse temperature}. The following theorem finds a gap between \ref{e:RLB} and \ref{BE} in the regime of small inverse temperature. 
\begin{theorem}\label{theorem:dyson} \ 
For the Dyson space $(\R^n,g,w_\beta)$ and $N\geq N_\beta:= n+\frac{\beta}{2}n(n-1)$,
\begin{itemize}
\item $\beta \ge 1 \implies \Ric_N\geq 0 \ \& \ \BE(0,N)$ hold;
\item $0<\beta<1$ $\implies$ $\Ric_N\geq 0$ holds while $\BE(K,\infty)$ does not for any $K \in \R$.
\end{itemize}
\end{theorem}

\paragraph{Idea of the proof.} 
The Ricci tensor bound $\Ric_N \ge 0$ can be immediately seen by a straightforward computation of the Bakry--\'Emery $N$-Ricci tensor for every $0<\beta<\infty$ and $N\geq N_\beta$:
\begin{equation}\label{eq:RicN}
    \begin{aligned}
        \Ric_N(v, v)_\x={}&\beta\sum_{i<j}\frac{(v_i-v_j)^2}{(x_i-x_j)^2}-\frac{\beta^2}{N-n}\left(\sum_{i<j}\frac{v_i-v_j}{x_i-x_j}\right)^2\\
        \geq {}&\left(\beta-\frac{\beta^2}{N-n}\cdot \frac{n(n-1)}{2}\right)\sum_{i<j}\frac{(v_i-v_j)^2}{(x_i-x_j)^2}\geq 0
    \end{aligned}
\end{equation}
for $v=(v_1,v_2,\dots,v_n)\in T_\x \R^n$.
This $N_\beta$ is sharp since the equality holds when $x_i=v_i$ for every $i=1,2,\dots,n$.
 Let $\Scal:=\bigcup_{i<j}\{x_i=x_j\}$ be the diagonal set, which is the singular points (zeros of the weight~$w_\beta$),
and $\Rcal:=\R^n\setminus \Scal$ the set of regular points.
We note that the $\Ric_\infty\geq 0$ holds on $\Rcal$,
since the singular potential $\log w_\beta=\beta\sum_{i<j}\log|x_i-x_j|$ is {\it locally} convex on $\Rcal$.

The difference in \ref{BE} between $\beta\in (0,1)$ and $\beta\in[1,\infty)$ occurs in their Sobolev spaces.
When $\beta\in[1,\infty)$,
then the $(1,2)$-capacity of $\Scal$ is zero (see Proposition \ref{prop:nullcapacity}), and we have the coincidence $H^{1,2}(\Rcal,w_\beta)=H^{1,2}(\R^n,w_\beta)$. In this case, the Sobolev space splits into the direct sum $\bigoplus_{\sigma \in \mathfrak S_n}H^{1,2}(X_\sigma, w_\beta)$, where $\sigma$ specifies each section $X_\sigma$ separated by the diagonal $\Scal$. Hence, the associated $L^2$-heat semigroup also splits as an iterative tensor product of the heat semigroup acting on each section $H^{1,2}(X_\sigma, w_\beta)$.  As each section supports \ref{BE} with $K=0$, this extends to the whole Sobolev space. In contrast, when $\beta\in (0,1)$,
$\Scal$ has a positive $(1,2)$-capacity. In this case, we can construct functions disproving the $(K,\infty)$-weak Bochner inequality \ref{wB}, which is equivalent to disproving~$\BE(K,\infty)$. 
More precisely, we can construct families of functions $u_r$ and $\phi_r$ such that, when $r$ approaches to $0$, the LHS of \ref{wB} goes to $-\infty$, while the RHS converges to $0$. The function $u_r$ will be constructed by cutting a \emph{formal} harmonic function off. When $\beta\in (0,1)$, the function $u_r$ is not locally Lipschitz and 
 $|\nabla u_r|^2$ does not lie in the domain of the Laplacian, which causes a gap between \ref{e:RLB} and \ref{BE}.
\paragraph{The statistical phisical viewpoint}
The Dyson space arises from what is called {\it Dyson Brownian motion}, which is the eigenvalue distributions of a Hermitian matrix valued Brownian motions introduced in \cite{Dys62} when~$\beta=2$. The Bakry--\'Emery curvature bound has played a significant role, e.g., to study local equilibrium in~\cite{ErdSchYau11}. From the statistical physical viewpoint, Theorem~\ref{theorem:dyson} states that the gap between $\Ric_{N} \ge 0$ and $\BE(0,N)$ appears if and only if the solution to the corresponding stochastic differential equation (called {\it Dyson SDE})
    $$\diff X_t^{i}=\frac{\beta}{2}\sum_{j: i \neq j}^n \frac{\diff t}{X_t^i-X_t^j}+\diff B_t^i \qquad i \in \{1,\ldots, n\}$$
    has collisions among particles within finite time, where $B_t^1, \ldots, B_t^n$ are independent Brownian motions in $\R$. See, e.g., \cite{CepaLep97, CepaLep01, AndGuiZei09} for the
    collision of the Dyson SDE. By applying It\^o's formula, the infinitesimal generator of the Dyson SDE is identical to the (weighted) Laplacian in the Dyson space~$(\R^n, w_\beta)$. Our result shows a phase transition regarding the Bakry--\'Emery gradient estiamte for the transition semigroup of the Dyson SDE in the small inverse temperature regime.

\paragraph{$(1,2)$-capacity vs $(2,2)$-capacity}
The two conditions~\ref{e:RLB} and \ref{BE} are bridged by what is called $\Gamma_2$-condition.
For a weighted Riemannian manifold $(M,g,e^{-V}{\rm vol}_g)$,
let $\Acal_0\subset L^2(M,e^{-V}{\rm vol}_g)$ be a dense subset.
We say that $(M,g,e^{-V}{\rm vol}_g)$ satisfies $\Gamma_2(K,N)$ if for every $u \in \Acal_0$, it holds
\begin{align}\tag*{$\Gamma_2(K,N)$} \label{G2}
\cdc_2(u) \ge K\cdc(u) +\frac{1}{N}({\rm L}u)^2 \comma
\end{align}
 where $\cdc(u,v):=\langle \nabla u, \nabla v\rangle$ is the square gradient operator and $\cdc_2(u, v):=\frac{1}{2}({\rm L} \cdc(u, v)- \cdc({\rm L} u, v) - \cdc(u, {\rm L} v))$ is {\it the $\cdc_2$-operator} with the infinitesimal generator~${\rm L}=\Delta- \nabla V\cdot\nabla$.
It is a well-known sufficient condition that, if the space $\Acal_0$ is dense in the domain of the infinitesimal generator (called {\it the essential self-adjointness} (ESA)), then \ref{BE} holds in the framework of spaces endowed with Markov triplet, see \cite[Cor.~3.3.19]{BakGenLed14}. 
Recall that the curvature involves the second-order differential structure and, having $\mathcal A_0=C_c^\infty(\mathcal R)$, the ESA imposes the negligibility of the singular set $\Scal$ in terms of the $(2,2)$-capacity, which is the second-order differential object as well. From this viewpoint,  it is not surprising that the ESA serves a sufficient condition for the equivalence~\ref{e:RLB} $\iff$ \ref{G2} $\iff$ \ref{BE}, see \cite[Thm.\ 1.1]{Wan11} and \cite[Cor.\ 2.3]{AmbGigSav15}.  A question that has not been fully understood is whether the negligibility in terms of the $(2,2)$-capacity is necessary. 
The contribution of this paper in this context is to provide the necessary and sufficient condition for this equivalence in the Dyson space, by which we clarify that the negligibility of the singular set in terms of {\it the $(1,2)$-capacity} (the first-order differential structure) is essential rather than the $(2,2)$-capacity.

\paragraph{Relation to $\RCD$ theory}
For every $\beta>0$,  the Dyson space does not satisfy the $\RCD(K,\infty)$ condition for any $K \in \R$, which can be easily seen by the lack of the Bruun--Minkowski inequality, see~\cite[Thm.\ 30.7]{Vil09} for the precise definition of the Bruun--Minkowski 
 inequality. 
The idea of the disproof is as follows: take two balls $A_0,A_1\subset \Rcal$ so that its intermediate subset $A_{1/2}=\{\frac{a_0}{2}+\frac{a_1}{2}\mid a_i\in A_i\}$ is centred in $\Scal$. Then this pair fails the Brunn--Minkowski inequality by letting their radii sufficiently small. 
This is not surprising because the Hamiltonian $V(\mathbf x)=-\beta \sum_{i<j}^n\log|x_i-x_j|$ is not convex in the whole $\R^n$, which is convex only on each section $X_\sigma$.
 We will not use this fact to prove/disprove \ref{BE}.



In \cite{HonSun25}, they provided a characterisation for almost smooth spaces to be $\RCD(K,N)$ spaces in terms of weighted Ricci tensor lower bounds on its regular part. 
Here an almost smooth space is a metric measure space that consists of a regular part equipped with a weighted Riemannian structure, and a singular part being {\it a set of zero $(1,2)$-capacity}. Our result confirms that the $(1,2)$-capacity condition is necessary, even to obtain $\BE$, which is  weaker than $\RCD$. 

\paragraph{Comparison with the unlabeled Dyson Brownian motion in the configuration space} Finally, we remark that in \cite{Suz23}, the first author proved $\BE(0,\infty)$ for 
the Dyson Brownian motion with {\it every} $\beta >0$ in the {\it configuration space}.  Indeed,  $\RCD(0, N_\beta)$ holds with $N_\beta=n+\frac{\beta}{2}n(n-1)$ thanks to \eqref{eq:RicN} and the geodesically convexity of the interior set.
This does not contradict Theorem~\ref{theorem:dyson} because the configuration space is the quotient space~$(\R^n/\mathfrak{S}_n,g,w)$  with respect to the symmetric group $\mathfrak S_n$, which regards all sections separated by $\Scal$ as a single space. As long as we look at a single section, the potential stays convex, by which $\BE(0, N_\beta)$ as well as $\RCD(0,N_\beta)$ remain true for all $\beta >0$. See Remark~\ref{rmk:boundary} for further technical details. 
This shows a difference between the {\it labelled} Dyson Brownian motion in~$\R^n$ and the {\it unlabelled} one in the configuration space from the viewpoint of the Bakry--\'Emery gradient estimate  when~$0<\beta<1$.

\section*{Acknowledgement}
The first author greatly appreciates the Theoretical Sciences Visiting Program (TSVP) at the Okinawa Institute of Science and
Technology for supporting his stay. He thanks Qing Liu and Xiaodan Zhou for their hospitality during his stay at OIST.
The authors appreciate Kazuhiro Kuwae for his comments on tamed spaces. They also thank Shouhei Honda, Shin-ichi Ohta and Nikita Evseev for fruitful discussions.

\section*{Data Availability Statement.}
No datasets were generated or analysed during the current study.

\section{\texorpdfstring{$\BE(K,N)$}{BE} holds when \texorpdfstring{$\beta\geq 1$}{bge1}}
\subsection{Capacity estimate} 
\begin{definition}[Weighted Sobolev space] Let $X \subset \R^n$ be an open subset. For $f \in C^\infty(X)$, define
\begin{align*}
\|u\|_{H^{1,2}(X)}^2:= \|u\|_{L^2(X)}^2 + \|\nabla u\|_{L^2(X)}^2 \fstop
\end{align*}
We define the {\it $(1,2)$-Sobolev space} $H^{1,2}(X, w_\beta)$ as the completion of $C^\infty(X) \cap \{f: \|f\|_{H^{1,2}(X)}<\infty\}$ with respect to the norm~$\|\cdot\|_{H^{1,2}(X)}$. When $X=\R^n$, we simply write $\|\cdot\|_{H^{1,2}}$. 
\end{definition} 
For a compact set $K\subset \R^n$, define 
    \[\Ascr(K):=\{g\in H^{1,2}(\R^n,w_\beta)\mid g\geq 1\text{ on a neighbourhood of }K\} \fstop \]
\begin{definition}[Capacity]
   The $(1,2)$-capacity$~\capa_\beta$ is defined as follows:
\begin{itemize} 
   \item For a compact set $K \subset \R^n$, 
    \[\capa_\beta(K)=\inf_{g\in \Ascr(K)}\|g\|_{H^{1,2}}\comma\]
    where $\inf \emptyset =+\infty$ conventionally. 
   \item  For an arbitrary subset $E\subset \R^n$, 
    \[\capa_\beta(E)=\sup_{K\subset E}\capa_\beta(K) \fstop \]
    \end{itemize}
 \end{definition}
We prove that the $(1,2)$-capacity of the singular set $\Scal$ is zero when $\beta\ge1$. For $i\neq j\in\{1,2,\dots,n\}$, we write the \emph{oriented} hyperplane $\Scal_{ij}=\{x_i=x_j\}$.
 \begin{proposition}\label{prop:nullcapacity}
     $\capa_\beta(\Scal)=0$ if $\beta\geq 1$.
 \end{proposition}
 \begin{proof}
The idea of the proof is as follows: For $s>0$,
 construct a family of Lipschitz functions $g_s$ such that $g_s\equiv 1$ around $\Scal$; cut them off by an appropriate function $\xi$; prove $\|g_s\cdot \xi\|_{H^{1,2}}\to 0$ as $s\to 0$. We now give the proof.
 \medskip
 
{\it Step 1. Parametrisation of $\R^n$.}
 Define the function $\di_\Scal:\R^n\to \R$ as
\[\di_\Scal(\x):=\di(\x,\Scal)=\min\{\di(\x,\y)\mid \y\in \Scal\}\fstop\]
 Then $\di_\Scal$ is a $1$-Lipschitz function.
For an ordered pair $(i,j)\in\{1,2,\dots,n\}^2$,
we say that a point $\x\in\Rcal$ belongs to $\Ucal_{ij}$ if there are  unique $h\in\Scal_{ij}$ and $t>0$ such that
\[\x=h+\frac{t}{\sqrt{2}}(\mathbf{e}_i-\mathbf{e}_j),~\text{ and }~\di_\Scal(\x)=t \comma\]
where $\mathbf{e}_k$ $(k=1,2,\dots,n)$ is the canonical unit vector.
Then $\Ucal:=\bigsqcup_{i\neq j}\Ucal_{ij}$ is a set of full measure in $\Rcal$.
On each $\Ucal_{ij}$, we can endow the local coordinates $(t,h)$ so that $t=\di_\Scal(\x)=\di(\x,\Scal_{ij})$. For a small number $s\in(0,1)$,
define the Lipschitz continuous function $g_s:\R^n\to\R$ as
\begin{align} \label{e:CO}
g_s(\x)=\begin{cases}
    1 & \di_\Scal(\x)\in(0,s),\\
    1+\log|\log(\di_\Scal(\x))|-\log|\log(s)| & \di_{\Scal}(\x)\in[s,s^{1/e}],\\
    0 & \di_\Scal(\x)\in(s^{1/e},\infty).
\end{cases}
\end{align}
Note that, on each $\Ucal_{ij}$, the function $g_s$ depends only on the variable $t=\di_\Scal(\x)$.

\medskip 
{\it Step 2. Cutting $g_s$ off.} For $r>0$,
    we denote by $rB^n\subset \R^n$  a closed ball of radius $r$ (with an unspecified centre).
     We will show that $\capa_\beta(\Scal\cap \frac13 B^n)=0$ for $\frac13 B^n$ with every centre.
For concentric balls $\frac13 B^n$ and $\frac12 B^n$,
let $\xi:\R^n\to \R$ be a cut-off function such that
\begin{equation}\label{eq:xi}
    \frac13 B^n\subset \supp(\xi)\subset \frac12 B^n,~~\xi\equiv 1~~\text{ on }\frac13 B^n,~~|\xi|,|\nabla\xi|\leq 1.
\end{equation}
Then $g_s\cdot \xi$ is a compactly supported Lipschitz function.
Furthermore,
$g_s\cdot \xi\equiv 1$ around a neighbourhood of $\Scal \cap \frac13 B^n$,
so $g_s\cdot \xi\in \Ascr(\Scal \cap \frac13 B^n)$.

 \medskip    
{\it Step 3. Computation of $\|g_s\cdot \xi\|_{H^{1,2}}$.}
By the definition, the support of $g_s\cdot \xi$ is contained in the product space as
\[\supp(g_s\cdot \xi)\cap \Ucal_{ij}\subset [0,s^{1/e}]\times \frac{1}{2}B^{n-1}_{ij},\]
where the centre of the ball $\frac12B^{n-1}_{ij}$ is the projection of the centre of $B^n$ on $\Scal_{ij}$.
Therefore we can estimate as
\begin{align}
    &\int_{\R^n}\left[|g_s\cdot \xi|^2+|\nabla (g_s\cdot \xi)|^2\right]w_\beta dx^{\otimes n}\notag\\
    =&\sum_{i,j}\int_{\Ucal_{ij}}\left[|g_s\cdot \xi|^2+|\nabla (g_s\cdot \xi)|^2\right]\prod_{k<l}|x_k-x_l|^\beta dtdh\notag\\
    \leq & \sum_{i,j}\int_0^{s^{1/e}}\int_{\frac12 B^{n-1}_{ij}}\left[|g_s\cdot \xi|^2+|\nabla (g_s\cdot \xi)|^2\right]\prod_{k<l}|x_k-x_l|^\beta  dtdh\label{eq:1-1}
    \end{align}
For every~$\x=(x_1,x_2,\dots,x_n)\in \frac12 B^n$ and $i,j\in\{1,2,\dots,n\}$,
we have $|x_i-x_j|\leq 1$.
Furthermore,
for $\x=(t,h)\in \Ucal_{ij}\cap\supp(g_s)$,
we have $|x_i-x_j|=\sqrt{2}t$.
Combined with~\eqref{eq:xi},
we can continue the estimate as
\begin{align}\label{eq:1-2}
    \eqref{eq:1-1}\leq& \sum_{i,j}\int_0^{s^{1/e}}\int_{\frac12 B^{n-1}_{ij}}\left[2|g_s|^2+|\nabla g_s|^2+2|g_s|\cdot|\nabla g_s|\right]|\sqrt{2}t|^\beta dtdh.
    \end{align}
    Since $g_s$ does not depend on $h$,
    we have
    \begin{align*}
    \eqref{eq:1-2}=& \sum_{i,j}\int_0^{s^{1/e}}\int_{\frac12 B^{n-1}_{ij}}\left[2|g_s|^2+\left|\pder[g_s]{t}\right|^2+2|g_s|\cdot\left|\pder[g_s]{t}\right|\right]|\sqrt{2}t|^\beta dtdh\\
    \leq 
    & \mathrm{vol}_{n-1}\left(\frac12 B^{n-1}\right)\cdot n(n-1)\cdot 2^{\beta/2}\int_0^{s^{1/e}}\left[2|g_s|^2+\left|\pder[g_s]{t}\right|^2+2|g_s|\cdot\left|\pder[g_s]{t}\right|\right]|t| dt,
\end{align*}
where ${\rm vol}_{n-1}$ denotes the (unweighted) Lebesgue measure in $\R^{n-1}$ and the assumption $\beta\geq 1$ is used for $|t|^\beta \le |t|$ for $0\le t \le 1$. 
The last integration can be computed explicitly with \eqref{e:CO},
and it converges to $0$ as $s\to 0$.
Therefore, $\capa_\beta(\Scal\cap\frac13 B^n)=0$ for $\frac13 B^n$ with every centre.
In view of the sub-additivity of the capacity, we conclude
$\capa_\beta(\Scal)=0$.
\end{proof}

\subsection{The proof of \texorpdfstring{$\BE(K,N)$}{BE} when \texorpdfstring{$\beta\geq 1$}{bge1}}
For an open set~$X \subset \R^n$,  the {\it $L^2$-semigroup}~$(T^X_t)_{t \ge 0}$ is defined as the propagator of the $L^2$-gradient flow, viz., $T^X_tf$ is the solution to
\begin{align*}
\partial_t u = - \nabla_{L^2}\mathcal E_X(u) \qquad u_0=f \in H^{1,2}(X, w_\beta) \comma
\end{align*}
where $\nabla_{L^2}$ is the Fr\'echet derivative in $L^2(X, w_\beta)$ and $\mathcal E_X$ is the Dirichlet energy given by
$$\mathcal E_X(f):=\frac{1}{2}\int_{X}|\nabla f|^2 w_\beta dx^{\otimes n} \fstop$$ 
Due to, e.g., \cite[Thm.~1.3.3]{Dav89}, the semigroup $T_t^{X}$ can be extended to $L^p(X, w_\beta)$ for $1 \le p \le \infty$. We simply write $T_t$ when $X=\R^n$.
\begin{theorem}
    For $\beta\geq 1$,
    the Dyson space $(\R^n,g,w_\beta)$ satisfies $\BE(0,N)$ for every $N \in [n+\frac{\beta n(n-1)}{2}, \infty]$. 
\end{theorem}

\begin{proof}
    From the definition of $\Scal$, $\R^n\setminus \Scal$ has a decomposition
    \[\R^n\setminus \Scal=\bigsqcup_{\sigma\in \mathfrak{S}_n} X_\sigma,\]
    where elements in the $n$-symmetric group $\mathfrak{S}_n$ corresponds to the signature of $x_i-x_j$.
In view of~\cite[Thm.\ 2.44]{HeiKilMar93} and Proposition~\ref{prop:nullcapacity}, the two Sobolev spaces are identical: $H^{1,2}(\R^n,w_\beta)=H^{1,2}(\R^n \setminus \Scal,w_\beta)$. Thus, the following decomposition holds:
    \begin{equation}\label{eq:sobolevequal}
        H^{1,2}(\R^n,w_\beta)=\bigoplus_{\sigma\in \mathfrak{S}_n}H^{1,2}(X_\sigma,w_\beta).
    \end{equation}
Similarly, we have the tensorisation of the semigroup 
\begin{align} \label{eq:decmoposesemigroup}
T_t=\bigotimes_{s \in \mathfrak S_n}T_t^{X_\sigma} \fstop
\end{align}
Due to the volume test~\cite[Theorem~4]{Stu94}, $T_t^{X_\sigma}$ is conservative, i.e., $T_t^{X_\sigma}\1=\1$ for every $\sigma \in \mathfrak S_n$. In particular, 
\begin{align} \label{eq:decmoposesemigroup2}
T_t u_\sigma = T_t^{X_\sigma}u_\sigma \prod_{\sigma' \neq \sigma}T_t^{X_{\sigma'}}\1 = T_t^{X_\sigma}u_\sigma \cquad u_\sigma \in H^{1,2}(X_\sigma, w_\beta)\fstop
\end{align}
Rewriting $\prod_{i<j}^n |x_i-x_j|^\beta=e^{-H_n}$, where $H_n(\mathbf x)=-\beta\sum_{i<j}\log|x_i-x_j|$, it can be easily seen that $H_n$ is convex in each component $X_\sigma$. Thus, $T_t^{X_\sigma}$ satisfies $\BE(0,N)$:  for $\sigma \in \mathfrak S_n$ and $ u \in H^{1,2}(X_\sigma, w_\beta)$,
    \begin{align} \label{i:BES}
    |\nabla T_t^{X_\sigma}u_\sigma|^2 +\frac{2t}{N}\left[(\Delta - \nabla V\cdot\nabla) T_t^{X_\sigma} u_\sigma\right]^2\le T_t^{X_\sigma}|\nabla u_\sigma|^2 \fstop
    \end{align}
    Having \eqref{eq:sobolevequal}--\eqref{i:BES}, 
    \begin{align*}
    &|\nabla T_t u|^2+\frac{2t}{N}\left[(\Delta - \nabla V\cdot\nabla) T_t u\right]^2 \\
    &= \Bigl|\sum_{\sigma \in \mathfrak S_n}\nabla T^{X_\sigma}_t u_\sigma\Bigr|^2+\frac{2t}{N}\left[\sum_{\sigma\in \mathfrak S_n}(\Delta - \nabla V\cdot\nabla) T_t^{X_\sigma} u_\sigma\right]^2\\
    &= \sum_{\sigma \in \mathfrak S_n}\Bigl|\nabla T^{X_\sigma}_t u_\sigma\Bigr|^2+\frac{2t}{N}\left[(\Delta - \nabla V\cdot\nabla) T_t^{X_\sigma} u_\sigma\right]^2\\
    &\le \sum_{\sigma \in \mathfrak S_n} T^{X_\sigma}_t\bigl| \nabla u_\sigma\bigr|^2 = T_t|\nabla u|^2
    \end{align*}
    for $u=\sum_{\sigma \in \mathfrak S_n} u_\sigma \in \oplus_{\sigma \in \mathfrak S_n}H^{1,2}(X_\sigma, w_\beta)$, where the second and the fourth equalities follow by~
    \[\quad |\nabla T_t^{X_\sigma}u_\sigma||\nabla T_t^{X_{\sigma'}}u_{\sigma'}|=0 \cquad |\nabla u_\sigma||\nabla u_{\sigma'}|=0 \cquad \sigma \neq \sigma'\fstop \qedhere\]
\end{proof}
\begin{remark}\label{rmk:sobtolip}
    Although the Dyson space $(\R^n,g,w_\beta)$ satisfies $\BE(0,N)$, it does not support the $\RCD(0,\infty)$ condition because the Sobolev-to-Lipschitz property does not hold. Namely, there is a function $f\in H^{1,2}(\R^n, w_\beta)$ with $|\nabla f|\leq 1$ a.e.~without a Lipschitz continuous representative.
    Indeed, take $F=\sum_{\sigma\in \mathfrak{S}_n}c_\sigma1_{X_\sigma}$ where $c_{\sigma}\neq c_{\sigma'}$ for $\sigma \neq \sigma'$. 
    After cutting $F$ off by a nice function,
    we can construct a compactly supported function $f \in \oplus_{\sigma \in \mathfrak S_n} H^{1,2}(X_\sigma, w_\beta)$  such that $|\nabla f| \le 1$ a.e.. However, $f$ cannot have a Lipschitz representative, as it has a discontinuity on each $S_{ij}$.
    A gap between $\BE$ and $\RCD$ was observed in \cite{Hon18} with a different example.
\end{remark}

\section{\texorpdfstring{$\BE(K,\infty)$}{BE} fails for \texorpdfstring{$\beta\in(0,1)$}{ble1}}

The goal of this section is to prove the following theorem.
\begin{theorem}\label{thm:noBE}
    For $\beta\in(0,1)$,
    the Dyson space $(\R^n,g,w_\beta)$ does not satisfy $\BE(K,\infty)$ for all $K\in\R$.
\end{theorem}

\begin{proof}
Thanks to~the equivalence between $\BE(K,\infty)$ and the $(K,\infty)$-weak Bochner inequality (see \cite[Cor.\ 2.3]{AmbGigSav15}), it suffices to disprove the following inequality for every $K \in \R$:  
    \begin{equation}\label{wB}
        \frac12\int_{\R^n}|\nabla u|^2\Delta \phi dx^{\otimes n}\geq \int_{\R^n}\left[\langle\nabla\Delta u,\nabla u\rangle + K|\nabla u|^2\right]\phi dx^{\otimes n}.\tag*{{\rm wB}$(K,\infty)$}
    \end{equation}  
    for every $u\in \Dcal(\Delta)$ with $\Delta u\in H^{1,2}(\R^n,w_\beta)$,
    and for every $\phi\in \Dcal(\Delta)$ with $\phi\geq 0$, $\phi,\Delta \phi\in L^\infty(\R^n,w_\beta)$.
    Here a function $f\in H^{1,2}(\R^n,w_\beta)$ is in the domain $\Dcal(\Delta)$ if there exists $h\in L^2(\R^n,w_\beta)$ satisfying
    \[\int_{\R^n}\langle \nabla f,\nabla g\rangle w_\beta dx^{\otimes n}=-\int_{\R^n}hgw_\beta dx^{\otimes n} \cquad g\in H^{1,2}(\R^n,w_\beta) \fstop\]
    Such a unique $h$ is denoted by $\Delta f$.
    
   The idea of the proof is that for $r\in(0,1)$,
we will construct families of functions $u_r$ and $\phi_r$ satisfying the following property:
for every~$K\in\R$,
there exists $\delta>0$ such that
if $r<\delta$,
then the pair $(u_r, \phi_r)$ fails \ref{wB}.

\medskip
{\it Step 1. Domains of $u_r$ and $\phi_r$.}
Let $I_{12}=\{(i,j)\in\N^2\mid 1\leq i<j\leq n,~(i,j)\neq (1,2)\}$ be the set of indices,
and let $S^1,S^2\subset \{x_1=x_2\}$ be bounded subsets of the singular hyperplane given by
\[S^m=\left\{(x_1,x_2,\dots,x_n)\middle| \begin{array}{l}
    x_1=x_2\in[ -m, m],\\
    x_i\in\left[5i-m,5i+m\right]\text{ for }i=3,4\dots,n
\end{array}\right\}.\]
    For $r\in(0,1)$ and $m=1,2$, let $D_r^m\subset\R^n$ be the compact set given by
    \[D_r^m=\{h+t\left(1,-1,0,\dots,0\right)\in\R^n\mid h\in S^m,~~t\in[-r,r]\}.\]
    Thanks to the assumption $r<1$,
    the set $D_r^m$ intersects with the singular hyperplane $\{x_i=x_j\}$ if and only if $(i,j)=(1,2)$.
     Indeed, for $\x\in D_r^m$,
it holds
    \[
    \begin{cases}
        |x_i-x_j|\geq 5j-m-(5i+m)=5(j-i)-2m\geq 1 & 3\leq i<j,\\
        |x_i-x_j|\geq 5j-m-(m+r)=5j-2m-r\geq 10 & i=1,2,~j\geq 3.
    \end{cases}
    \]
    We frequently use the parametrisation $(t,h)$ on $D_r^m$ (and on $\R^n$ by the natural extension).
    Note that the equality $t=2(x_1-x_2)$ holds in this parametrisation, in particular, 
\begin{equation}\label{eq:tisx1-x2}
        \pder[]{t}=\pder[]{x_1}-\pder[]{x_2}.
    \end{equation}

\medskip    
{\it Step 2. Construct a function $u_r$.}
    For a given $r\in(0,1)$,
    define functions $f,\eta_r:\R^n\to \R$ by
    \begin{equation}\label{eq:harmonic}
    f_r(\x)=\prod_{i<j}|x_i-x_j|^{-\beta}(x_i-x_j),\end{equation}
    and
    \begin{equation}\label{eq:eta}
        \eta_r(\x)=\eta_r(t,h)=P_r(t)Q(h),
    \end{equation}
    where $P_r:\R\to \R$ and $Q:\{x_1=x_2\}\to \R$ be smooth functions such that
    \[P_r(t)=\begin{cases}
        1 & t\in[-r,r]\\
        0 & t\notin [-2r,2r]
    \end{cases},~~
    Q_r(h)=\begin{cases}
        1 & h\in S^1\\
        0 & h\notin S^2
    \end{cases}.\]
    Then $\eta_r$ is a smooth cut-off function i.e.
    \[\eta_r(\x)=\begin{cases}
        1 & \x\in D_r^1,\\
        0 & \x\notin D_{2r}^{2}.
    \end{cases}\]
    Define the function $u_r=f\cdot \eta_r$, which is a differentiable function such that its support intersects with the hyperplane $\{x_i=x_j\}$ if and only if $(i,j)=(1,2)$.
    

\medskip
 {\it Step 3. Show $u_r\in H^{1,2}(\R^n, w_\beta)$.}
 We only need to check the behaviour around the singular hyperplane $\{x_1=x_2\}$ since the support of $u_r$ is contained in $D_{2r}^2$, which does not intersects with the other hyperplanes $\{x_i=x_j\}$\footnote{We can verify $u_r\in W^{1,2}(\R^n,w_\beta)\implies u_r\in H^{1,2}(\R^n,w_\beta)$ from the well known fact that the weight $|x_1-x_2|^\beta dx^{\otimes n}$ belongs to Muckenhoupt's $\Acal_2$ class for $\beta\in (0,1)$.}.
    The gradient of $u_r$ can be  computed as,
   for $\x\in \R^n\setminus\{x_1=x_2\}$,
    \begin{equation}\label{eq:nablausquare-1}
    \begin{aligned}
        |\nabla u_r|^2(\x)=&\sum_{k=1}^n\Bigg[\sum_{p\neq k}\frac{1-\beta}{x_k-x_p}\prod_{i<j}|x_i-x_j|^{-\beta}(x_i-x_j)\eta_r\\
        &\qquad+\prod_{i<j}|x_i-x_j|^{-\beta}(x_i-x_j)\pder[\eta_r]{x_k}\Bigg]^2\\
        =&2(1-\beta)^2\eta_r^2|x_1-x_2|^{-2\beta}\\
        &\qquad+\theta_1(\x)|x_1-x_2|^{-2\beta}(x_1-x_2)+\theta_2(\x)|x_1-x_2|^{2-2\beta},
    \end{aligned}
    \end{equation}
    where $\theta_1,\theta_2$ are smooth functions supported on $D_{2r}^2$.
    Since $\beta\in(0,1)$, and $\eta_r,\theta_1,\theta_2$ are compactly supported smooth functions, we have
   \begin{align*}
        &\int_{\R^n}|\nabla u_r|^2(\x)w_\beta(\x)dx^{\otimes n}\\
        =&\int_{\R^n}\Big[2(1-\beta)^2\eta
    _r^2|x_1-x_2|^{-\beta}\\
    &\qquad+\theta_1(\x)|x_1-x_2|^{-\beta}(x_1-x_2)+\theta_2(\x)|x_1-x_2|^{2-\beta}\Big]\prod_{(i,j)\in I_{12}}|x_i-x_j|^\beta dx^{\otimes n}<+\infty.
   \end{align*}
As $u_r \in L^2(\R^n, w_\beta)$ as well, we conclude $u_r\in H^{1,2}(\R^n, w_\beta)$.

\medskip
{\it Step 4. $u_r\in \Dcal(\Delta)$.}
For a test function $\psi\in C^\infty(\R^n) \cap H^{1,2}(\R^n, w_\beta)$,
we have
\begin{align}
    &\int_{\R^n}\langle \nabla u_r,\nabla \psi\rangle w_\beta(\x) dx^{\otimes n}\notag\\
    =&\int_{\R^n}\left[\sum_{k=1}^n\left(\pder[f_r]{x_k}\eta_r+f_r\pder[\eta_r]{x_k}\right)\pder[\psi]{x_k}\right]\prod_{i<j}|x_i-x_j|^\beta dx^{\otimes n}\notag\\
    =&\int_{\R^n}\sum_{k=1}^n\left[\sum_{p\neq k}\frac{1-\beta}{x_k-x_p}\prod_{i<j}(x_i-x_j)\eta_r+\prod_{i<j}(x_i-x_j)\pder[\eta_r]{x_k}\right]\pder[\psi]{x_k}dx^{\otimes n}.\label{eq:2-1}
    \end{align}
    By using the integration by parts in terms of the unweighted Lebesgue measure and the fact that $\eta_r$ is compactly supported,
    \begin{equation}\label{eq:integrationbyparts}
    \begin{aligned}
\eqref{eq:2-1}=&-\int_{\R^n}\sum_{k=1}^n\Bigg[\sum_{q\neq p,k}\sum_{p\neq k}\frac{1-\beta}{(x_k-x_p)(x_k-x_q)}\eta_r\\
    &\qquad\qquad\qquad +\sum_{p\neq k}\frac{2-\beta}{x_k-x_p}\pder[\eta_r]{x_k}+\ppder[\eta_r]{x_k}\Bigg]\prod_{i<j}(x_i-x_j)\psi dx^{\otimes n}.
\end{aligned}
\end{equation}
Note that, by a straightforward computation, we have the identity
\begin{equation*}\label{eq:cancel}
    \sum_{q\neq p,k}\sum_{p\neq k}\frac{1-\beta}{(x_k-x_p)(x_k-x_q)}\prod_{i<j}(x_i-x_j)\equiv 0.
\end{equation*}
Hence, $\Delta u_r \in L^2(\R^n, w_\beta)$ exists and it is explicitly written by
\begin{equation}\label{eq:lapu-1}
    \Delta u_r = \sum_{k=1}^n\left[\sum_{p\neq k}\frac{2-\beta}{x_k-x_p}\pder[\eta_r]{x_k}+\ppder[\eta_r]{x_k}\right]\prod_{i<j}|x_i-x_j|^{-\beta}(x_i-x_j).
\end{equation}
Since $\eta_r\equiv 1$ on $D_r^1$,
this formula implies that $\Delta u_r\equiv 0$ on $D_r^1$.

\medskip
{\it Step 5. Show $\Delta u_r \in H^{1,2}(\R^n, w_\beta)$.}
As in Step 3,  we only need to discuss the order of the term $|x_1-x_2|$ in \eqref{eq:lapu-1} when $|x_1-x_2|$ approaches $0$.
Here we use the properties \eqref{eq:tisx1-x2} and \eqref{eq:eta}.
By extracting the $|x_1-x_2|$-terms in \eqref{eq:lapu-1},
\begin{align*}
    \Delta u_r &= (2-\beta)\left(\pder[\eta_r]{x_1}-\pder[\eta_r]{x_2}\right)|x_1-x_2|^{-\beta}+O(|x_1-x_2|^{1-\beta})\\
    &=(2-\beta)\pder[P_r]{t}(t)Q(h)|x_1-x_2|^{-\beta}+O(|x_1-x_2|^{1-\beta}).
\end{align*}
Since $P_r(t)\equiv 1$ around $t=0$ i.e. $x_1-x_2=0$,
the leading term vanishes around $t=0$,
and we obtain
\[\Delta u_r=O(|x_1-x_2|^{1-\beta}).\]
In the same way, we have
\[|\nabla\Delta u_r|^2=O(|x_1-x_2|^{-2\beta}),\]
therefore $\Delta u_r \in H^{1,2}(\R^n, w_\beta)$.

\medskip
{\it Step 6. Construct a function $\phi_r$.}
We now construct a function $\phi_r:\R^n\ni (t,h)\mapsto \phi_r(t,h) \in\R$ with its support contained in $D_{r}^1$.
Let $\Phi_r:\R\to\R$ be the $C^{1,1}$ function given by
\[
   \Phi_r(t)=\begin{cases}
    1+\cos(\frac{3\pi}{2r}t) & |t|\leq \frac{2r}{3},\\
    0 & |t|\geq \frac{2r}{3}.
\end{cases}\] 
Since $\Phi_r$ does not depends on $x_3,x_4,\dots,x_n$, we have
\begin{equation}\label{eq:zeroderivative-1}
    \pder[\Phi_r]{x_3}=\pder[\Phi_r]{x_4}=\cdots = \pder[\Phi_r]{x_n}=0.
\end{equation}
We now take $\Psi:\{x_1=x_2\}\to \R$ to be a positive smooth function such that
\[
    \supp(\Psi)\subset S^1.
\]
Since $\Psi$ does not depends on $t=x_1-x_2$,
we have
\begin{equation}\label{eq:zeroderivative-2}
    \pder[\Psi]{t}=\pder[\Psi]{x_1}-\pder[\Psi]{x_2}=0.
\end{equation}
We now define the function $\phi_r:\R^n\to \R $ as
\[\phi_r(t,h)=\Phi_r(t)\Psi(h).\]
Then $\phi_r$ is a $C^{1,1}$ function such that $\supp(\phi_r)\in D_r^1$.
By construction,  
we can easily check 
$\phi_r\geq 0$ and $\phi_r\in L^\infty(\R^n,w_\beta)$.

\medskip
{\it Step 7. Show $\phi_r\in\Dcal(\Delta)$ and $\Delta\phi_r\in L^\infty(\R^n,w_\beta)$.}
For a test function $\xi\in C^\infty(\R^n) \cap H^{1,2}(\R^n, w_\beta)$,
we have
\begin{align*}
&\int_{\R^n}\langle \nabla \phi_r,\nabla \xi\rangle w_\beta(\x) dx^{\otimes n}\\
    =&\int_{\R^n}\left[\sum_{k=1}^n\left(\pder[\Phi_r]{x_k}\Psi+\Phi_r\pder[\Psi]{x_k}\right)\pder[\xi]{x_k}\right]\prod_{i<j}|x_i-x_j|^\beta dx^{\otimes n}\\
    =&\int_{\R^n}\sum_{k=1}^n\left[\left(\pder[\Phi_r]{x_k}\Psi+\Phi_r\pder[\Psi]{x_k}\right)\prod_{i<j}|x_i-x_j|^\beta\right] \pder[\xi]{x_k}dx^{\otimes n}\\
    \overset{(\ast\ast)}{=}&-\int_{\R^n}\sum_{k=1}^n\Bigg[\sum_{p\neq k}\frac{\beta}{x_k-x_p}\left(\pder[\Phi_r]{x_k}\Psi+\Phi_r\pder[\Psi]{x_k}\right)\\
    &\qquad\qquad\qquad +\left(\ppder[\Phi_r]{x_k}\Psi+2\pder[\Phi_r]{x_k}\pder[\Psi]{x_k}+\Phi_r\ppder[\Psi]{x_k}\right)\Bigg]\prod_{i<j}|x_i-x_j|^\beta\xi dx^{\otimes n}
\end{align*}
Again we used the integration by parts on the unweighted Lebesgue measure at $(\ast\ast)$.
As $\Delta\phi_r$ defined as above is bounded and compactly supported, which will be seen just below, it is in particular in $L^2(\R^n, w_\beta)$. Thus, we have $\phi_r\in \Dcal(\Delta)$.
In view of \eqref{eq:zeroderivative-1} and \eqref{eq:zeroderivative-2},
we obtain
\begin{equation}\label{e:EP}
\begin{aligned} 
    \Delta\phi_r=&\left[\ppder[\Phi_r]{x_1}+\ppder[\Phi_r]{x_2}+\frac{\beta}{x_1-x_2}\cdot \left(\pder[\Phi_r]{x_1}-\pder[\Phi_r]{x_2}\right)\right]\Psi \\    
    &\quad +\left[\sum_{p\neq 1}\frac{\beta}{x_1-x_p}\pder[\Phi_r]{x_1}+\sum_{q\neq 2}\frac{\beta}{x_2-x_q}\pder[\Phi_r]{x_2}\right]\Psi\\
    &\qquad+\Phi_r\left[\sum_{(k,p)\in J_{12}}\frac{\beta}{x_k-x_p}\cdot\pder[\Psi]{x_k}+\sum_{k=1}^n\ppder[\Psi]{x_k}\right]\comma
\end{aligned}
\end{equation}
where $J_{12}=\{(k,p)\in \{1,2,\dots,n\}^2\mid (k,p)\neq (1,2), (2,1)\}$.
To check that $\Delta \phi_r$ is bounded,
we only need to care about the first term (denoted by $(I)$) in \eqref{e:EP}.
It can be explicitly written around $\{x_1=x_2\}$ as
\[(I)=\left[-\frac{18\pi^2}{r^2}\cos\left(\frac{3\pi}{r}(x_1-x_2)\right)-\frac{6\pi\beta}{r(x_1-x_2)}\cdot \sin\left(\frac{3\pi}{r}(x_1-x_2)\right)\right]\Psi.\]
Thus $\Delta \phi_r$ is bounded around the singular hyperplane $\{x_1=x_2\}$,
and also on the whole $\R^n$.

\medskip
{\it Step 8. Failure of \ref{wB}.}
In this step, we will see that the above functions $u_r,\phi_r$ disprove the weak Bochner inequality \ref{wB} for sufficiently small $r$. 
Due to~$\supp(\phi_r)\subset D_r^1$,
the expression of $|\nabla u_r|^2$ in \eqref{eq:nablausquare-1} is simplified as
\[
    |\nabla u_r|^2=(1-\beta)^2\sum_{k=1}^n\left[\sum_{p\neq k}\frac{1}{x_k-x_p}\right]^2\prod_{i<j}|x_i-x_j|^{2-2\beta}.
\]
 Since $|x_i-x_j|$ is bounded below uniformly in $r$, the term~$\frac{1}{x_i-x_j}$ is bounded uniformly in $r$ for $(i,j)\in I_{12}$ (recall $I_{12}:=\{(i,j)\in\N^2\mid 1\leq i<j\leq n,~(i,j)\neq (1,2)\}$. Recalling $\frac{1}{|x_1-x_2|}=\frac{1}{2|t|}\geq \frac{1}{2r}$, 
we can find $C_1(r)>1$ and a function $U_1:S^1\to [0,+\infty)$, which is independent of $t$ and $r$, such that
\begin{equation}\label{eq:approx1}
    C_1(r)^{-1}U_1(h)|t|^{-2\beta}\leq |\nabla u_r|^2(t,h)\leq C_1(r)U_1(h)|t|^{-2\beta} ~~\text{for}~~ (t,h)\in D_r^1
\end{equation}
and
\begin{equation}\label{converge1-1}
    C_1(r)\to 1~~\text{as}~~r\to 0.
\end{equation}
By a similar argument,
there are $C_2(r)>1$ and $U_2:S^1\to [0,+\infty)$ such that
\begin{equation}\label{eq:approx2}
    C_2(r)^{-1}U_2(h)|t|^\beta\leq w_\beta(t,h)\leq C_2(r)U_2(h)|t|^\beta~~\text{for}~~ (t,h)\in D_r^1
\end{equation}
and
\begin{equation}\label{converge1-2}
    C_2(r)\to 1~~\text{as}~~r\to 0.
\end{equation}
The expression $\Delta \phi_r$ in \eqref{e:EP} is explicitly written by
\begin{align*} \label{e:EP2}
    \Delta\phi_r(t,h)=&\left[-\frac{18\pi^2}{r^2}\cos\left(\frac{3\pi}{2r}t\right)-\frac{2\beta}{t}\cdot \frac{6\pi}{r}\sin\left(\frac{3\pi}{2r}t\right)\right]\Psi(h) \\ 
    &\quad +\frac{3\pi}{r}\left[-\sum_{p\neq 1}\frac{\beta}{x_1-x_p}\sin\left(\frac{3\pi}{2r}t\right)+\sum_{q\neq 2}\frac{\beta}{x_2-x_q}\sin\left(\frac{3\pi}{2r}t\right)\right]\Psi(h)\\
    &\qquad+\left(1+\cos\left(\frac{3\pi}{2r}t\right)\right)\left[\sum_{(k,p)\in J_{12}}\frac{\beta}{x_k-x_p}\cdot\pder[\Psi]{x_k}(h)+\sum_{k=1}^n\ppder[\Psi]{x_k}(h)\right].\end{align*}
    Since $\Psi$ is a compactly supported smooth function, we can find a positive constant $A_1$, which is independent of $t$ and $r$, such that
    \begin{equation}\label{eq:approx3}
        \Delta\phi_r(t,h)\leq -\left[\frac{18\pi^2}{r^2}\cos\left(\frac{3\pi}{2r}t\right)+\frac{12\pi\beta}{rt}\sin\left(\frac{3\pi}{2r}t\right)\right]\Psi(h)+\frac{A_1}{r}~~\text{for}~~(t,h)\in D_r^1.
    \end{equation}
Let us define the 
subset $I\subset [-\frac{2r}{3},\frac{2r}{3}]$ by
\[t\in I~~\iff~~\frac{18\pi^2}{r^2}\cos\left(\frac{3\pi}{2r}t\right)+\frac{12\pi\beta}{rt}\sin\left(\frac{3\pi}{2r}t\right)\geq 0,\]
and denote by $I^c=[-\frac{2r}{3},\frac{2r}{3}]\setminus I$ its complement.
By
using \eqref{eq:approx1}, \eqref{eq:approx2} and \eqref{eq:approx3},
we can compute the left hand side of \ref{wB} as
\begin{align*}
    &\frac12\int_{\R^n}|\nabla u|^2\Delta \phi_r w_\beta dx^{\otimes n}\\
    ={}&\frac12\int_{D_r^1}|\nabla u|^2\Delta \phi_r w_\beta dtdh\\
   \leq  {}& -\frac{1}{C_1(r)C_2(r)}\int_I\int_{S^1}\left[\frac{9\pi^2}{r^2}\cos\left(\frac{3\pi}{2r}t\right)+\frac{6\pi\beta}{rt}\sin\left(\frac{3\pi}{2r}t\right)\right]\Psi U_1U_2|t|^{-\beta}dtdh\\
   & -C_1(r)C_2(r)\int_{I^c}\int_{S^1}\left[\frac{9\pi^2}{r^2}\cos\left(\frac{3\pi}{2r}t\right)+\frac{6\pi\beta}{rt}\sin\left(\frac{3\pi}{2r}t\right)\right]\Psi U_1U_2|t|^{-\beta}dtdh\\
   &+C_1(r)C_2(r)\int_{-\frac{2r}{3}}^{\frac{2r}{3}}\int_{S^1}\frac{A_1}{r}U_1U_2|t|^{-\beta}dtdh.
   \end{align*}
   Furthermore, by changing the variable $s=\frac{3\pi}{2r}t$,
   we can continue as
   \begin{equation}\label{eq:3}
   \begin{aligned}
   =  {}& -r^{-1-\beta}\cdot \frac{A_2}{C_1(r)C_2(r)}\int_{\frac{3\pi}{2r}I}\left[\cos(s)+\beta\frac{\sin(s)}{s}\right]|s|^{-\beta}ds\\
   &-r^{-1-\beta}\cdot C_1(r)C_2(r)A_2\int_{\frac{3\pi}{2r}I^c}\left[\cos(s)+\beta\frac{\sin(s)}{s}\right]|s|^{-\beta}ds\\
   & +r^{-\beta}\cdot C_1(r)C_2(r)A_1\left(\frac{2}{3\pi}\right)^{1-\beta}\int_{S^1}U_1U_2dh\int_{-\pi}^{\pi}|s|^{-\beta}ds,
   \end{aligned}
   \end{equation}
   where $\frac{3\pi}{2r}I$ and $\frac{3\pi}{2r}I^c$ are naturally scaled domains and $A_2$ is the explicit constant given by
   \[A_2:=9\pi^2\left(\frac{2}{3\pi}\right)^{1-\beta}\int_{H^1}\Psi(h) U_1(h)U_2(h)dh.\]
   Finally, by the asymptotics \eqref{converge1-1} and \eqref{converge1-2},
   we obtain the asymptotic estimate of the three lines \eqref{eq:3} as
   \begin{align*}
    \eqref{eq:3}\leq & -r^{-1-\beta}\cdot A_2\int_{-\pi}^{\pi}\left[\cos(s)+\beta\frac{\sin(s)}{s}\right]|s|^{-\beta}ds+o(r^{-1-\beta})~~~~(r\to 0).
\end{align*}
The constant of the leading term is negative, which can be readily seen by the symmetry of trigonometric functions and the monotonicity of $|s|^{-\beta}$. Due to $0<\beta<1$, the leading term (i.e., the LHS of \ref{wB}) goes to $-\infty$ as $r \to 0$.

Regarding the RHS of \ref{wB},
for $K\leq 0$, we can compute as,
\begin{align*}
&\int_{\R^n}\left[\langle\nabla\Delta u_r,\nabla u_r\rangle + K|\nabla u_r|^2\right]\phi_r w_\beta dx^{\otimes n}\\
\overset{(\star)}{=} {} & \int_{D_r^1} K|\nabla u_r|^2\phi_r w_\beta dx^{\otimes n}\\
\geq{} & K\int_{-\frac{2r}{3}}^{\frac{2r}{3}}\int_{S^1}C_1U_1|t|^{-2\beta}\cdot 2\cdot C_2U_2|t|^{\beta}dtdh\\
={} & r^{1-\beta}\cdot 2C_1C_2K\int_{S^1}U_1U_2dh\left(\frac{2}{3\pi}\right)^{1-\beta}\int_{-\pi}^\pi|s|^{-\beta}ds\\
\to{} & 0~~~~(r\to 0).
\end{align*}
Here the equality $(\star)$ holds since $\Delta u_r\equiv 0$ on $D_r^1$.
Therefore, as $r$ tends to $0$, the RHS converges to $0$.
Thus, for any $K\in \R$,
we can take a small $r$ so that the function $u_r,\phi_r$ does not support \ref{wB}.
\end{proof}
 \begin{remark}\label{rmk:betoql}
The function $f$ constructed in \eqref{eq:harmonic} is a locally integrable harmonic function. If the Dyson space supports the local weak Poincar\'e inequality,  \cite[Thm.\ 1.1]{Jia14} would provide another way to disprove $\BE(0,\infty)$. 
\end{remark}

\begin{remark}\label{rmk:tamed}
By a similar proof, we can also disprove $\BE(\kappa, \infty)$ with a distributional lower bound $\kappa$ in the extended Kato class $\mathcal{K}_{-1}$ in the sense of \cite{ErbRigStuTam22}. Indeed, 
by \eqref{eq:nablausquare-1},
the function $u_r$ is in $\Dcal(\Delta)$ while $|\nabla u_r|$ is not in $H^{1,2}(\R^n,w_\beta)$.
Hence, the conclusion of \cite[Prop.\ 6.10]{ErbRigStuTam22} does not hold.
\end{remark}

\begin{remark}\label{rmk:boundary}
        In \cite{Suz23},
        the first author proved that the Dyson Brownian motion on the configuration space $(\R^n/\mathfrak{S}_n,g,w_\beta)$ satisfies $\BE(0,\infty)$.
        Our argument cannot be applied to the configuration space.
        Indeed, in the configuration space, we have to use Stokes' Theorem {\it with a boundary term} in the integration by parts~\eqref{eq:integrationbyparts}
        because the gradient $\nabla u_r$ does not necessarily vanish along the direction of normal vectors in $\Scal$.
        This prevents $u_r$ from lying in $\Dcal(\Delta)\subset H^{1,2}(\R^n/\mathfrak{S}_n,w_\beta)$.
        
        
\end{remark}

\bibliographystyle{alpha}
\bibliography{MasterBib.bib}

\end{document}